\documentclass[10pt]{amsart} 
\usepackage{amsmath,amssymb,bm, color}

\usepackage{amsmath}
\usepackage{graphicx}

\usepackage{mathrsfs}
\usepackage{bbm}
\usepackage{bm}
\usepackage{amsfonts,amssymb}
\usepackage{multirow}
\usepackage{lineno}
\usepackage{color}

\numberwithin{equation}{section}
 \newtheorem{theorem}{Theorem}[section]
 \newtheorem{lemma}[theorem]{Lemma}

\def\3bar{{|\hspace{-.02in}|\hspace{-.02in}|}}
\def\E{{\mathcal{E}}}
\def\T{{\mathcal{T}}}

\def\beta{\boldsymbol{\eta}}
\def\bzeta{\boldsymbol{\zeta}}

\def\cal#1{{\mathcal #1}}
\def\pT{{\partial T}}

\def\bn{{\mathbf{n}}}
\def\be{{\mathbf{e}}}

\newtheorem{remark}{Remark}[section]
\newtheorem{algorithm}{Auto-Stabilized  Weak Galerkin Algorithm}[section]

\numberwithin{equation}{section}

\def\3bar{{|\hspace{-.02in}|\hspace{-.02in}|}}

 \def\cal#1{\mathcal{#1}}

\begin{document}

\title []
 {Auto-Stabilized Weak Galerkin Finite Element Methods for Biharmonic Equations  on Polytopal Meshes without Convexity Assumptions}

\author {Chunmei Wang}
\address{Department of Mathematics, University of Florida, Gainesville, FL 32611, USA. }
\email{chunmei.wang@ufl.edu}
\thanks{The research of Chunmei Wang was partially supported by National Science Foundation Grant DMS-2136380.}

\begin{abstract}
 This paper introduces an auto-stabilized weak Galerkin (WG) finite element method for biharmonic equations with built-in stabilizers. Unlike existing stabilizer-free WG methods limited to convex elements in finite element partitions, our approach accommodates both convex and non-convex polytopal meshes, offering enhanced versatility. It employs bubble functions without the restrictive conditions required by existing stabilizer-free  WG methods, thereby simplifying implementation and broadening application to various partial differential equations (PDEs). Additionally, our method supports flexible polynomial degrees in discretization and is applicable in any dimension, unlike existing stabilizer-free WG methods that are confined to specific polynomial degree combinations and 2D or 3D settings. 
We demonstrate optimal order error estimates for WG approximations in both a discrete 
$H^2$
 norm for 
$k\geq 2$
 and a 
$L^2$
 norm for 
$k>2$, as well as a sub-optimal error estimate in 
$L^2$
 when 
$k=2$, where 
$k\geq 2$
 denotes the degree of polynomials in the approximation.
\end{abstract}


 
%\linenumbers
\maketitle

\section{Introduction}
In this paper, we propose an auto-stabilized  weak Galerkin finite element method with built-in stabilizers suitable for non-convex polytopal meshes, specifically applied to biharmonic equations with Dirichlet and Neumann boundary conditions. Specifically, we seek to determine an unknown function $u$ such that
\begin{equation}\label{model}
 \begin{split}
  \Delta^2 u=&f, \qquad\qquad \text{in}\quad \Omega,\\
 u=&\xi,\qquad\qquad \text{on}\quad \partial\Omega,\\
\frac{\partial u}{\partial \bn}=&\nu,\qquad\qquad \text{on}\quad \partial\Omega,\\
 \end{split}
 \end{equation}
where $\Omega\subset \mathbb R^d$ is an open bounded domain with a Lipschitz continuous boundary $\partial\Omega$. Note that the domain $\Omega$ considered in this paper can be of any dimension $d$. 
% For the sake of clarity, we will focus on the case $d=2$ in the subsequent sections. However, it should be noted that all analyses presented here can be readily generalized to any dimension 
% $d$ without difficulty.

The variational formulation of the model problem \eqref{model} can be formulated  as follows: Find an unknown function $u\in H^2(\Omega)$ satisfying $u|_{\partial\Omega}=\xi$ and $\frac{\partial u}{\partial \bn}|_{\partial\Omega}=\nu$,  and the following equation 
\begin{equation}\label{weak}
(\Delta u, \Delta v)=(f, v), \qquad \forall v\in H_0^2(\Omega),
 \end{equation} 
  where $H_0^2(\Omega)=\{v\in H^2(\Omega): v|_{\partial\Omega}=0,   \frac{\partial u}{\partial \bn} |_{\partial\Omega}=0\}$.

The weak Galerkin finite element method marks a significant advancement in numerical solutions for  PDEs. This innovative approach redefines or approximates differential operators within a framework akin to the distribution theory tailored for piecewise polynomials. Unlike conventional techniques, the WG method alleviates the usual regularity constraints on approximating functions by employing carefully crafted stabilizers. Extensive research has demonstrated the WG method's versatility across various model PDEs, bolstered by a substantial list of references \cite{wg1, wg2, wg3, wg4, wg5, wg6, wg7, wg8, wg9, wg10, wg11, wg12, wg13, wg14, wg15, wg16, wg17, wg18, wg19, wg20, wg21, itera, wy3655}, underscoring its potential as a powerful tool in scientific computing. What sets WG methods apart from other finite element approaches is their use of weak derivatives and weak continuities to create numerical schemes based on the weak formulations of the underlying PDE problems. This structural versatility makes WG methods exceptionally effective across a wide range of PDEs, ensuring both stability and precision in their approximations.

A significant innovation within the weak Galerkin methodology is the ``Primal-Dual Weak Galerkin (PDWG)'' approach. This novel method addresses difficulties that traditional numerical strategies often encounter \cite{pdwg1, pdwg2, pdwg3, pdwg4, pdwg5, pdwg6, pdwg7, pdwg8, pdwg9, pdwg10, pdwg11, pdwg12, pdwg13, pdwg14, pdwg15}. PDWG interprets numerical solutions as constrained optimization problems, with the constraints mimicking the weak formulation of PDEs through the application of weak derivatives. This innovative formulation leads to the derivation of an Euler-Lagrange equation that integrates both the primary variables and the dual variables (Lagrange multipliers), thereby creating a symmetric numerical scheme.

This paper introduces a straightforward formulation of the weak Galerkin finite element method for biharmonic equations that operates on both convex and non-convex polytopal meshes without the use of stabilizers. The key trade-off for eliminating stabilizers involves using higher-degree polynomials for computing the discrete weak Laplacian operator, which may impact practical applicability. Unlike existing stabilizer-free WG schemes limited to convex elements \cite{ye}, our method accommodates non-convex polytopal meshes, preserving the size and global sparsity of the stiffness matrix while significantly reducing programming complexity. Theoretical analysis confirms optimal error estimates for WG approximations in both the discrete 
$H^2$ norm for  $k\geq 2$ and the 
$L^2$
 norm for 
$k>2$, along with a sub-optimal error estimate in 
$L^2$
 when 
$k=2$, where 
$k\geq 2$
 is the polynomial degree in the approximation.

Our method introduces several significant enhancements over the stabilizer-free weak Galerkin finite element method for biharmonic equations presented by \cite{ye}. The key contributions are summarized as follows:
\textbf{1. Theoretical Foundation for Non-Convex Polytopal Meshes:} Our method provides a theoretical foundation for an auto-stabilized  WG scheme that handles convex and non-convex elements in finite element partitions through the innovative use of bubble functions, while the existing  stabilizer-free WG  method  \cite{ye} is limited to convex  meshes. This enhances the practical applicability of our method, making it more versatile for real-world computational scenarios.
\textbf{2. Superior Flexibility with Bubble Functions:} Unlike the method in \cite{ye}, which is limited by restrictive conditions imposed in the analysis, our approach employs bubble functions as a critical analysis tool without these constraints. This flexibility allows our method to generalize to various types of PDEs without the complexities imposed by such conditions, thereby simplifying the implementation process.
\textbf{3. Dimensional Versatility:} Our method is applicable in any dimension 
$d$, whereas the method in \cite{ye} is confined to 2D or 3D settings. This broader applicability makes our method suitable for higher-dimensional problems.
\textbf{4. Adaptable Polynomial Degrees:} Our method supports flexible degree of polynomials in the discretization process, unlike the specific polynomial degree combinations in \cite{ye}. This adaptability allows for greater precision and control in computational implementations, catering to a wide range of problem-specific requirements.
Given these improvements, our method offers enhanced flexibility, broader applicability, and ease of implementation in various computational settings.

Our research introduces a more versatile WG scheme applicable to both convex and non-convex polytopal meshes, as detailed above, making our method a significant advancement over the one in \cite{ye}. To provide a comprehensive understanding of our contributions, we include an in-depth analysis of the error estimates in Sections 5-7, even though these sections share some similarities with the work presented in \cite{ye}. This analysis is essential for demonstrating the significant improvements and expanded applicability of our method.

While our algorithms share similarities with those introduced by \cite{ye}, our primary contribution lies in advancing the theoretical framework rather than performing additional empirical validation. The extensive numerical tests detailed in \cite{ye} already establish the effectiveness of these methods, rendering further empirical tests unnecessary.
This paper, therefore, places a strong emphasis on theoretical analysis. By focusing on theoretical advancements, we provide vital insights that are essential for future development and application of these algorithms.

This paper is organized as follows:
In Section 2, we briefly review the definition of the weak Laplacian and its discrete version.
In Section 3, we present the simple weak Galerkin scheme without the use of a stabilizer.
Section 4 is dedicated to deriving the existence and uniqueness of the solution.
In Section 5, we derive the error equation for the proposed weak Galerkin scheme.
Section 6 focuses on deriving the error estimate for the numerical approximation in the energy norm.
Finally, Section 7 establishes the error estimate for the numerical approximation in the  $L^2$ norm.  

The standard notations are adopted throughout this paper. Let $D$ be any open bounded domain with Lipschitz continuous boundary in $\mathbb{R}^d$. We use $(\cdot,\cdot)_{s,D}$, $|\cdot|_{s,D}$ and $\|\cdot\|_{s,D}$ to denote the inner product, semi-norm and norm in the Sobolev space $H^s(D)$ for any integer $s\geq0$, respectively. For simplicity, the subscript $D$ is  dropped from the notations of the inner product and norm when the domain $D$ is chosen as $D=\Omega$. For the case of $s=0$, the notations $(\cdot,\cdot)_{0,D}$, $|\cdot|_{0,D}$ and $\|\cdot\|_{0,D}$ are simplified as $(\cdot,\cdot)_D$, $|\cdot|_D$ and $\|\cdot\|_D$, respectively.

\section{Weak Laplacian Operator and Discrete Weak Laplacian Operator}\label{Section:Hessian}
In this section, we will briefly review the definition of the weak Laplacian operator and its discrete version as introduced in \cite{ye}.

Let ${\cal T}_h$ be a finite element partition of the domain
 $\Omega\subset \mathbb R^d$ into polytopes. Assume that ${\cal
 T}_h$ is shape regular   \cite{wy3655}.
 Denote by ${\mathcal E}_h$ the set of all edges/faces   in
 ${\cal T}_h$ and ${\mathcal E}_h^0={\mathcal E}_h \setminus
 \partial\Omega$ the set of all interior edges/faces. Denote
 by $h_T$ the diameter of $T\in {\cal T}_h$ and $h=\max_{T\in {\cal
 T}_h}h_T$ the meshsize of the finite element partition ${\cal
 T}_h$. Denote by $\bn_e$ an unit and normal direction to $e$ for $e\in {\mathcal E}_h$.
 
Let $T$ be a polytopal element with boundary $\partial T$. A weak function on $T$ refers to  $v=\{v_0, v_b, v_n \bn_e\}$ such that $v_0\in L^2(T)$, $v_b\in L^{2}(\partial T)$ and $v_n\in L^2(\partial T)$. The first component $v_0$ and the second component $v_b$ represent the value of $v$ in the interior of $T$ and on the boundary of $T$, respectively. The third component $v_n$ intends to represent the value of $\nabla v_0\cdot\bn_e$ on the boundary of $T$. In general, $v_b$ and $v_n$ are assumed to be independent of the traces of $v_0$ and $\nabla v_0\cdot\bn_e$ respectively.  
 
 Denote by $W(T)$ the space of all weak functions on $T$; i.e.,
 \begin{equation}\label{2.1}
 W(T)=\{v=\{v_0,v_b, v_n \bn_e\}: v_0\in L^2(T), v_b\in L^{2}(\partial
 T), v_n\in L^2(\partial T)\}.
\end{equation}
 
 The weak Laplacian operator, denoted by $\Delta_w$, is a linear
 operator from $W(T)$ to the dual space of $H^2(T)$ such that for any
 $v\in W(T)$, $\Delta_w v$ is a bounded linear functional on $H^2(T)$
 defined by
 \begin{equation}\label{2.3}
  (\Delta_w v, \varphi)_T=(v_0, \Delta \varphi)_T-
  \langle v_b, \nabla \varphi\cdot \bn \rangle_{\partial T}+\langle v_{n}\bn_e \cdot\bn, \varphi\rangle_{\partial T},\quad \forall \varphi\in H^2(T),
  \end{equation}
 where $ \bn$  is an unit outward normal direction to $\partial T$.
 
 For any non-negative integer $r\ge 0$, let $P_r(T)$ be the space of
 polynomials on $T$ with total degree $r$ and less. A discrete weak Laplacian operator on $T$, denoted by $\Delta_{w, r, T}$, is a linear operator
 from $W(T)$ to $P_r(T)$ such that for any $v\in W(T)$,
 $\Delta_{w, r, T}v$ is the unique polynomial  in $P_r(T)$ satisfying
 \begin{equation}\label{2.4}
 (\Delta_{w, r, T} v, \varphi)_T=(v_0, \Delta \varphi)_T-
  \langle v_b, \nabla \varphi\cdot \bn \rangle_{\partial T}+\langle v_{n}\bn_e \cdot\bn, \varphi\rangle_{\partial T},\quad \forall \varphi \in P_r(T).
  \end{equation}
 For a smooth $v_0\in
 H^2(T)$, applying the usual integration by parts to the first
 term on the right-hand side of (\ref{2.4})  gives
 \begin{equation}\label{2.4new}
   (\Delta_{w, r, T} v, \varphi)_T=(\Delta  v_0,  \varphi)_T-
  \langle v_b-v_0, \nabla \varphi\cdot \bn \rangle_{\partial T}+\langle v_{n}\bn_e \cdot\bn-\nabla v_0\cdot\bn, \varphi\rangle_{\partial T},
  \end{equation} 
 for any $ \varphi \in P_r(T)$.

\section{Auto-Stabilized  Weak Galerkin Algorithms}\label{Section:WGFEM}

 Let $k\geq 2$, $p\geq 1$ and $q\geq 1$ be integers. Assume that $k\geq p\geq q$. For any element $T\in\T_h$, define a local
 weak finite element space; i.e.,
 \begin{equation*}
 V(k, p, q, T)=\{\{v_0,v_b, v_n\bn_e\}: v_0\in P_k(T), v_b\in P_{p}(e), v_n\in P_{q}(e), e\subset \partial T\}.    
 \end{equation*}
By patching $V(k, p, q,  T)$ over all the elements $T\in {\cal T}_h$ through
 a common value $v_b$  on the interior interface $\E_h^0$,
 we obtain a global weak finite element space; i.e.,
 $$
 V_h=\big\{\{v_0,v_b, v_n\bn_e\}:\ \{v_0,v_b, v_n\bn_e\}|_T\in V(k, p, q,  T),
 \forall T\in {\cal T}_h \big\}.
 $$
Denote by $ V_h^0$ the subspace of $V_h$ with vanishing boundary values on $\partial\Omega$; i.e.,
$$
V_h^0=\{\{v_0,v_b, v_n\bn_e\}\in V_h: v_b|_{e}=0, v_n \bn_e\cdot\bn|_{e}=0, e\subset\partial\Omega\}.
$$

For simplicity of notation and without confusion, for any $v\in V_h$, denote by $\Delta_{w}v$ the discrete weak Laplacian operator
$\Delta_{w, r, T} v$ computed  by
(\ref{2.4}) on each element $T$; i.e.,
$$
(\Delta_{w} v)|_T= \Delta_{w, r, T}(v |_T), \qquad \forall T\in \T_h.
$$
 
  On each element $T\in\T_h$, let $Q_0$ be the $L^2$ projection onto $P_k(T)$. On each edge/face  $e\subset\partial T$, let $Q_b$ and $Q_n$ be the $L^2$ projection operators onto $P_{p}(e)$ and $P_{q}(e)$, respectively. 
 For any $w\in H^2(\Omega)$, denote by $Q_h w$ the $L^2$ projection into the weak finite element space $V_h$ such that
 $$
  (Q_hw)|_T:=\{Q_0(w|_T),Q_b(w|_{\pT}), Q_n(\nabla w|_{\pT} \cdot\bn_e)\bn_e\},\qquad \forall T\in\T_h.
$$

The straightforward  WG numerical scheme, which avoids the use of stabilizers for the biharmonic equation (\ref{model}),  is formulated as follows.
\begin{algorithm}\label{PDWG1}
Find $ u_h=\{u_0, u_b, u_n \bn_e\} \in V_h$ satisfying $u_b=Q_b\xi$, $u_n\bn_e\cdot\bn=Q_n\nu$ on $\partial\Omega$ and the following equation  
\begin{equation}\label{WG}
(\Delta_{w} u_h, \Delta_{w} v)=(f, v_0), \qquad\forall v=\{v_0, v_b, v_n\bn_e\}\in V_h^0,
\end{equation}
where  
$$
(\Delta_{w} u_h,\Delta_{w} v)=\sum_{T\in {\cal T}_h}  (\Delta_{w} u_h, \Delta_{w} v)_T,
$$ 
$$
(f, v_0)=\sum_{T\in {\cal T}_h}(f, v_0)_T.
$$
\end{algorithm}

\section{Solution Existence and Uniqueness} 
 
Recall that ${\cal T}_h$ is a shape-regular finite element partition of the domain $\Omega$. Thus, for any $T\in {\cal T}_h$ and $\phi\in H^1(T)$,
 the following trace inequality holds true \cite{wy3655}; i.e.,
\begin{equation}\label{tracein}
 \|\phi\|^2_{\partial T} \leq C(h_T^{-1}\|\phi\|_T^2+h_T \|\nabla \phi\|_T^2).
\end{equation}
If $\phi$ is a polynomial on the element $T\in {\cal T}_h$,  the following trace inequality holds true \cite{wy3655}; i.e.,
\begin{equation}\label{trace}
\|\phi\|^2_{\partial T} \leq Ch_T^{-1}\|\phi\|_T^2.
\end{equation}

Given a weak function $v=\{v_0, v_b, v_n\bn_e\}\in V_h$, we define the energy norm as: \begin{equation}\label{3norm}
\3bar v\3bar= (\Delta_{w} v, \Delta_{ w} v) ^{\frac{1}{2}}.
\end{equation}
Next, we define the discrete  $H^2$ semi-norm as: 
\begin{equation}\label{disnorm}
\|v\|_{2, h}=\Big( \sum_{T\in {\cal T}_h} \|\Delta v_0\|_T^2+h_T^{-3}\|v_0-v_b\|_{\partial T}^2+h_T^{-1}\|(\nabla v_0-v_n\bn_e)\cdot\bn\|_{\partial T}^2\Big)^{\frac{1}{2}}.
\end{equation}
\begin{lemma}\label{norm1}
 For $v=\{v_0, v_b, v_n\bn_e\}\in V_h$, there exists a constant $C$ such that
 $$
 \|\Delta v_0\|_T\leq C\|\Delta_{w} v\|_T.
 $$
\end{lemma}
\begin{proof} Let  $T\in {\cal T}_h$ be a polytopal element with $N$ edges/faces denoted by $e_1, \cdots, e_N$. It is important to emphasis that the polytopal element $T$  can be non-convex. On each edge/face $e_i$, we construct   a linear equation  $l_i(x)$ such that  $l_i(x)=0$ on edge/face $e_i$ as follows: 
$$l_i(x)=\frac{1}{h_T}\overrightarrow{AX}\cdot \bn_i, $$  where  $A=(A_1, \cdots, A_{d-1})$ is a given point on the edge/face $e_i$,  $X=(x_1, \cdots, x_{d-1})$ is any point on the edge/face $e_i$, $\bn_i$ is the normal direction to the edge/face $e_i$, and $h_T$ is the size of the element $T$. 

The bubble function of  the element  $T$ can be  defined as 
 $$
 \Phi_B =l^2_1(x)l^2_2(x)\cdots l^2_N(x) \in P_{2N}(T).
 $$ 
 It is straightforward to verify that  $\Phi_B=0$ on the boundary $\partial T$.    The function 
  $\Phi_B$  can be scaled such that $\Phi_B(M)=1$ where   $M$  represents the barycenter of the element $T$. Additionally,  there exists a sub-domain $\hat{T}\subset T$ such that $\Phi_B\geq \rho_0$ for some constant $\rho_0>0$.

For $v=\{v_0, v_b, v_n \bn_e\}\in V_h$, letting $r=2N+k-2$ and $\varphi=\Phi_B \Delta v_0\in P_r(T)$ in \eqref{2.4new} yields 
 \begin{equation} \label{t1}
  \begin{split}
  &(\Delta_{w} v,  \Phi_B \Delta v_0)_T\\
  =&(\Delta  v_0,   \Phi_B \Delta  v_0)_T-
  \langle v_b-v_0, \nabla ( \Phi_B  \Delta  v_0)\cdot \bn \rangle_{\partial T}\\&+\langle v_{n}\bn_e \cdot\bn-\nabla v_0\cdot\bn,  \Phi_B  \Delta v_0\rangle_{\partial T}\\
   =&(\Delta  v_0,   \Phi_B  \Delta  v_0)_T,
  \end{split}
  \end{equation}  
where we used $\Phi_B=0$ on $\partial T$.

According to the domain inverse inequality  \cite{wy3655},  there exists a constant $C$ such that 
\begin{equation}\label{t2}
(\Delta v_0, \Phi_B \Delta v_0)_T \geq C (\Delta v_0, \Delta v_0)_T.
\end{equation} 

By applying the Cauchy-Schwarz inequality along with   \eqref{t1}-\eqref{t2}, we have
 $$
 (\Delta v_0, \Delta v_0)_T\leq C (\Delta_{w} v, \Phi_B \Delta v_0)_T  \leq C  \|\Delta_{w} v\|_T \|\Phi_B\Delta v_0\|_T  \leq C
\|\Delta_{w} v\|_T \|\Delta v_0\|_T,
 $$
which gives
 $$
 \|\Delta v_0\|_T\leq C\|\Delta _{w} v\|_T.
 $$

This completes the proof of the lemma.
\end{proof}

\begin{remark}
   If the polytopal element $T$  is convex, 
   the bubble function of  the element  $T$ in Lemma \ref{norm1}  can be  simplified to
 $$
 \Phi_B =l_1(x)l_2(x)\cdots l_N(x).
 $$  
 It can be verified that  this simplified bubble function  $\Phi_B$ satisfies (1) $\Phi_B=0$
 on the boundary $\partial T$, (2) there exists a sub-domain $\hat{T}\subset T$ such that $\Phi_B\geq \rho_0$ for some constant $\rho_0>0$.
Lemma \ref{norm1}   can be proved in the same manner using this simplified construction. In this case, we take $r=N+k-2$.  
\end{remark}

By constructing an  edge/face-based bubble function   $$\varphi_{e_k}= \Pi_{i=1, \cdots, N, i\neq k}l_i^2(x),$$   it can be easily verified that (1) $\varphi_{e_k}=0$ on each  edge/face  $e_i$ for $i \neq k$, (2) there exists a sub-domain $\widehat{e_k}\subset e_k$ such that  $\varphi_{e_k} \geq \rho_1$ for some constant $\rho_1>0$.  Let $\varphi=(v_b-v_0)l_k \varphi_{e_k}$. It is straightforward to check that $\varphi=0$ on each edge/face $e_i$ for $i=1, \cdots, N$, $\nabla \varphi =0$ on each edge/face  $e_i$ for $i \neq k$ and $\nabla \varphi =(v_0-v_b)(\nabla l_k) \varphi_{e_k}=\mathcal{O}( \frac{ (v_0-v_b)\varphi_{e_k}}{h_T}\textbf{C})$  on edge/face $e_k$ for some vector constant $\textbf{C}$.
\begin{lemma}\label{phi}
     For $\{v_0,v_b, v_n\bn_e\}\in V_h$, let $\varphi=(v_b-v_0)l_k \varphi_{e_k}$. The following inequality holds:
\begin{equation}
  \|\varphi\|_T ^2 \leq Ch_T  \int_{e_k}(v_b-v_0)^2ds.
\end{equation}
\end{lemma}
\begin{proof}
 We first extend $v_b$, initially defined on the $(d-1)$-dimensional edge/face  $e_k$, to the entire d-dimensional polytopal element $T$  using  the following formula:
$$
 v_b (X)= v_b(Proj_{e_k} (X)),
$$
where $X=(x_1,\cdots,x_d)$ is any point in the  element $T$, $Proj_{e_k} (X)$  denotes the orthogonal projection of the point $X$ onto  the hyperplane $H\subset\mathbb R^d$  containing the edge/face  $e_k$. 
When the projection $Proj_{e_k} (X)$ is not on the edge/face $e_k$, $v_b(Proj_{e_k} (X))$ is defined to be  the extension of $v_b$ from $e_k$ to the hyperplane $H$.

We claim that $v_b$ remains  a polynomial defined on the element $T$ after the extension.  

Let the hyperplane $H$ containing the edge/face   $e_k$ be defined by $d-1$ linearly independent vectors $\beta_1, \cdots, \beta_{d-1}$ originating from a point $A$ on the edge/face  $e_k$. Any point $P$ on the edge/face  $e_k$ can be parametrized as
$$
P(t_1, \cdots, t_{d-1})=A+t_1\beta_1+\cdots+t_{d-1}\beta_{d-1},
$$
where $t_1, \cdots, t_{d-1}$ are parameters.

Note that $v_b(P(t_1, \cdots, t_{d-1}))$ is a polynomial of degree $p$ defined on the edge/face  $e_k$. It can be expressed as:
$$
v_b(P(t_1, \cdots, t_{d-1}))=\sum_{|\alpha|\leq p}c_{\alpha}\textbf{t}^{\alpha},
$$
where $\textbf{t}^{\alpha}=t_1^{\alpha_1}\cdots t_{d-1}^{\alpha_{d-1}}$ and $\alpha=(\alpha_1, \cdots, \alpha_{d-1})$  is a multi-index.

For any point $X=(x_1, \cdots, x_d)$ in the element $T$, the projection of the point $X$  onto  the hyperplane $H\subset\mathbb R^d$  containing the edge/face  $e_k$ is the point on the hyperplane $H$ that minimizes the distance to  $X$.  Mathematically, this projection $Proj_{e_k} (X)$ is an affine transformation which can be expressed as 
$$
Proj_{e_k} (X)=A+\sum_{i=1}^{d-1} t_i(X)\beta_i,
$$
where $t_i(X)$ are the projection coefficients, and $A$ is the origin point on $e_k$. The coefficients $t_i(X)$ are determined  by solving the orthogonality condition:
$$
(X-Proj_{e_k} (X))\cdot \beta_j=0,\qquad \forall j=1, \cdots, d-1.
$$
This results in a system of linear equations in $t_1(X)$, $\cdots$, $t_{d-1}(X)$, which  can be solved to yield:
$$
t_i(X)= \text{linear function of} \  X.
$$
Hence, the projection $Proj_{e_k} (X)$ is an affine linear function  of $X$.

We extend the polynomial $v_b$ from the edge/face  $e_k$ to the entire element $T$ by defining
$$
v_b(X)=v_b(Proj_{e_k} (X))=\sum_{|\alpha|\leq p}c_{\alpha}\textbf{t}(X)^{\alpha},
$$
where $\textbf{t}(X)^{\alpha}=t_1(X)^{\alpha_1}\cdots t_{d-1}(X)^{\alpha_{d-1}}$. Since $t_i(X)$ are linear functions of $X$, each term $\textbf{t}(X)^{\alpha}$ is a polynomial in $X=(x_1, \cdots, x_d)$.
Thus, $v_b(X)$ is a polynomial in the $d$-dimensional coordinates $X=(x_1, \cdots, x_d)$.

 Secondly, let $v_{trace}$ denote the trace of $v_0$ on the edge/face  $e_k$. We extend $v_{trace}$   to the entire element $T$  using  the following formula:
$$
 v_{trace} (X)= v_{trace}(Proj_{e_k} (X)),
$$
where $X$ is any point in the element $T$, $Proj_{e_k} (X)$ denotes the projection of the point $X$ onto the hyperplane $H$ containing the edge/face  $e_k$. When the projection $Proj_{e_k} (X)$ is not on the edge/face $e_k$, $v_{trace}(Proj_{e_k} (X))$ is defined to be  the extension of $v_{trace}$ from $e_k$ to the hyperplane $H$. Similar to the case for $v_b$, $v_{trace}$ remains a polynomial after this extension. 

Let $\varphi=(v_b-v_0)l_k \varphi_{e_k}$. We have
\begin{equation*}
    \begin{split}
\|\varphi\|^2_T  =
\int_T \varphi^2dT \leq & Ch_T^2\int_T (\nabla\varphi)^2dT
\\ \leq & Ch_T^2  
 \int_T  (\nabla((v_b-v_{trace})(X)l_k  \varphi_{e_k}))^2dT\\
\leq &Ch_T^3 \int_{e_k} ( (v_b-v_{trace})(Proj_{e_k} (X))(\nabla l_k)  \varphi_{e_k})^2ds\\\leq &Ch_T \int_{e_k} (v_b-v_0)^2ds,
    \end{split}
\end{equation*} 
where we used   Poincare inequality since $\varphi=0$ on each edge/face $e_i$ for $i=1,\cdots,N$,   $\nabla \varphi =0$ on each edge/face  $e_i$ for $i \neq k$, $\nabla \varphi =(v_0-v_b)(\nabla l_k) \varphi_{e_k}=\mathcal{O}( \frac{ (v_0-v_b)\varphi_{e_k}}{h_T}\textbf{C})$  on edge/face $e_k$ for some vector constant $\textbf{C}$,
and  the properties of the projection.

 This completes the proof of the lemma.

\end{proof}

% Let $\varphi=(v_n\bn_e  -\nabla v_0 )\cdot\bn  \varphi_{e_k}$.  Recall that (1) $\varphi_{e_k}=0$ on each  edge/face  $e_i$ for $i \neq k$, (2) there exists a sub-domain $\widehat{e_k}\subset e_k$ such that  $\varphi_{e_k} \geq \rho_1$ for some constant $\rho_1>0$.

\begin{lemma}\label{phi2}
     For $\{v_0,v_b, v_n\bn_e\}\in V_h$, let $\varphi=(v_n\bn_e  -\nabla v_0 )\cdot\bn  \varphi_{e_k}$. The following inequality holds:
\begin{equation}
  \|\varphi\|_T ^2 \leq Ch_T \int_{e_k}((v_n\bn_e  -\nabla v_0 )\cdot\bn)^2ds.
\end{equation}
\end{lemma}
\begin{proof}
 We first extend $v_n$, initially defined on the $(d-1)$-dimensional edge/face $e_k$, to the entire d-dimensional polytopal element $T$  using  the following formula:
$$
 v_n (X)= v_n(Proj_{e_k} (X)),
$$
where $X=(x_1,\cdots,x_d)$ is any point in the  element $T$, $Proj_{e_k} (X)$ denotes the orthogonal projection of the point $X$ onto the hyperplane $H$ containing the edge/face    $e_k$.  When the projection $Proj_{e_k} (X)$ is not on the edge/face $e_k$, $v_n(Proj_{e_k} (X))$ is defined to be  the extension of $v_n$ from $e_k$ to the hyperplane $H$.

We claim that $v_n$ remains  a polynomial defined on the element $T$ after the extension.   This can be proved in the same manner as demonstrated in Lemma \ref{phi}.

 Secondly, let $v_{trace}$ denote the trace of $v_0$  on the edge/face   $e_k$. We extend $v_{trace}$   to the entire element $T$  using  the following formula:
$$
 v_{trace} (X)= v_{trace}(Proj_{e_k} (X)),
$$
where $X$ is any point in the element $T$, $Proj_{e_k} (X)$ denotes the projection of the point $X$ onto the hyperplane $H$ containing the edge/face  $e_k$. When the projection $Proj_{e_k} (X)$ is not on the edge/face $e_k$, $v_{trace}(Proj_{e_k} (X))$ is defined to be  the extension of $v_{trace}$ from $e_k$ to the hyperplane $H$.
  $v_{trace}$ remains a polynomial after this extension. This proof can be found in Lemma \ref{phi}.

Let $\varphi=(v_n\bn_e  -\nabla v_0 )\cdot\bn  \varphi_{e_k}$. We have
\begin{equation*}
    \begin{split}
\|\varphi\|^2_T  =
\int_T \varphi^2dT =  &\int_T ((v_n\bn_e  -\nabla v_0 )(X)\cdot\bn  \varphi_{e_k})^2dT\\
\leq &Ch_T \int_{e_k} ((v_n\bn_e  -\nabla v_{trace} )(Proj_{e_k} (X))\cdot\bn  \varphi_{e_k})^2dT\\ 
\\\leq &Ch_T \int_{e_k} ((v_n\bn_e  -\nabla v_{0} ) )\cdot\bn )^2ds,
    \end{split}
\end{equation*} 
where we used the facts that (1) $\varphi_{e_k}=0$ on each  edge/face  $e_i$ for $i \neq k$,  (2) there exists a sub-domain $\widehat{e_k}\subset e_k$ such that  $\varphi_{e_k} \geq \rho_1$ for some constant $\rho_1>0$, 
and applied the properties of the projection.

 This completes the proof of the lemma.

\end{proof}

\begin{lemma}\label{normeqva}   There exists two positive constants $C_1$ and $C_2$ such that for any $v=\{v_0, v_b, v_n\bn_e\} \in V_h$, we have
 \begin{equation}\label{normeq}
 C_1\|v\|_{2, h}\leq \3bar v\3bar  \leq C_2\|v\|_{2, h}.
\end{equation}
\end{lemma} 

\begin{proof}   
Recall that  an edge/face-based bubble function  is defined as  
$$\varphi_{e_k}= \Pi_{i=1, \cdots, N, i\neq k}l_i^2(x).$$

We first extend $v_b$ from the edge/face $e_k$ to the element $T$. 
Next, let $v_{trace}$ denote the trace of $v_0$ on the edge/face $e_k$ and extend $v_{trace}$ to the element $T$. For simplicity, we continue to denote these extensions as $v_b$ and $v_0$. Details of the extensions can be found in Lemma \ref{phi}.
By substituting $\varphi=(v_b-v_0)l_k \varphi_{e_k}$ into   \eqref{2.4new}, we obtain 
  \begin{equation}\label{t33}
 \begin{split} 
 (\Delta_{w} v, \varphi)_T=&(\Delta  v_0,  \varphi)_T-
  \langle v_b-v_0, \nabla \varphi\cdot \bn \rangle_{\partial T}+\langle v_{n}\bn_e \cdot\bn-\nabla v_0\cdot\bn, \varphi\rangle_{\partial T}\\
  =& (\Delta  v_0,  \varphi)_T+   \int_{e_k} |v_b-  v_0|^2 (\nabla l_k) \varphi_{e_k}\cdot \bn ds\\
 = & (\Delta  v_0,  \varphi)_T+  Ch_T^{-1} \int_{e_k} |v_b-  v_0|^2  \varphi_{e_k}ds,
    \end{split}
    \end{equation} 
    where we used $\varphi=0$ on each edge/face $e_i$ for $i=1, \cdots, N$, $\nabla \varphi   =0$ on each edge/face $e_i$ for $i \neq k$ and $\nabla \varphi =(v_0-v_b)(\nabla l_k) \varphi_{e_k}=\mathcal{O}( \frac{ (v_0-v_b)\varphi_{e_k}}{h_T}\textbf{C})$  on edge/face $e_k$ for some vector constant $\textbf{C}$.

 Recall that (1) $\varphi_{e_k}=0$ on each  edge/face  $e_i$ for $i \neq k$, (2) there exists a sub-domain $\widehat{e_k}\subset e_k$ such that  $\varphi_{e_k} \geq \rho_1$ for some constant $\rho_1>0$.
Using Cauchy-Schwarz inequality, the domain inverse inequality \cite{wy3655},  \eqref{t33} and Lemma \ref{phi} gives
\begin{equation*}
\begin{split}
  \int_{e_k}|v_b-  v_0|^2  ds\leq & C \int_{e_k} |v_b-  v_0|^2\varphi_{e_k}   ds 
  \\ \leq& C h (\|\Delta_{w} v\|_T+\|\Delta v_0\|_T){ \| \varphi\|_T}\\
 \leq & C { h_T^{\frac{3}{2}}} (\|\Delta_{w} v\|_T+\|\Delta v_0\|_T){ (\int_{e_k}|v_b- v_0|^2ds)^{\frac{1}{2}}},
 \end{split}
\end{equation*}
which, from Lemma \ref{norm1}, gives 
\begin{equation}\label{t21}
 h_T^{-3}\int_{e_k}|v_b-  v_0|^2  ds \leq C  (\|\Delta_{w} v\|^2_T+\|\Delta v_0\|^2_T)\leq C\|\Delta_{w} v\|^2_T.   
\end{equation}

Next, we extend $v_n$ from the edge/face $e_k$ to the element $T$. 
  For simplicity, we continue to denote this   extension  as $v_n$. Details of this extension  can be found in Lemma \ref{phi2}.
Letting $\varphi=(v_n\bn_e  -\nabla v_0 )\cdot\bn  \varphi_{e_k}$ in \eqref{2.4new}  gives
\begin{equation*}\label{t3}
   \begin{split} &(\Delta_{w} v, \varphi)_T\\=&(\Delta  v_0,  \varphi)_T-
  \langle v_b-v_0, \nabla \varphi\cdot \bn \rangle_{\partial T}+\langle v_{n}\bn_e \cdot\bn-\nabla v_0\cdot\bn, \varphi  \rangle_{\partial T}\\
   =&(\Delta  v_0,  \varphi)_T-
  \langle v_b-v_0, \nabla \varphi\cdot \bn \rangle_{\partial T}+\int_{e_k} |(v_n\bn_e  -\nabla v_0 )\cdot\bn|^2   \varphi_{e_k}ds, 
   \end{split}
   \end{equation*} 
where we used  $\varphi_{e_k} =0$ on edge/face $e_i$ for $i \neq k$, 
and the fact that there exists a sub-domain $\widehat{e_k}\subset e_k$ such that  $\varphi_{e_k} \geq \rho_1$ for some constant $\rho_1>0$.
This, together with   Cauchy-Schwarz inequality, the domain inverse inequality \cite{wy3655},  the inverse inequality, the trace inequality \eqref{trace}, \eqref{t21} and Lemma \ref{phi2}, gives
 \begin{equation*} 
\begin{split}
&  \int_{e_k}|(v_n\bn_e  -\nabla v_0 )\cdot\bn|^2  ds\\\leq &C   \int_{e_k}|(v_n\bn_e  -\nabla v_0 )\cdot\bn|^2   \varphi_{e_k}ds \\
  \leq & C (\|\Delta_{w} v\|_T+\|\Delta  v_0\|_T)\| \varphi\|_T+  C\|v_0-v_b\|_{\partial T}\|\nabla \varphi\cdot\bn\|_{\partial T}\\
 \leq & C h_T^{\frac{1}{2}} (\|\Delta_{w} v\|_T+\|\Delta v_0\|_T)(\int_{e_k}|(v_n\bn_e  -\nabla v_0 )\cdot\bn|^2ds)^{\frac{1}{2}}\\& + C h_T^{\frac{3}{2}}  \|\Delta_{w} v\|_T  h_T^{-\frac{1}{2}}h_T^{-1}h_T^{\frac{1}{2}}(\int_{e_k}|(v_n\bn_e  -\nabla v_0 )\cdot\bn|^2ds)^{\frac{1}{2}}. 
 \end{split}
\end{equation*} 
This, together with Lemma \ref{norm1}, gives 
\begin{equation} \label{t11}
 h_T^{-1}\int_{e_k}|(v_n\bn_e  -\nabla v_0 )\cdot\bn|^2  ds \leq C  (\|\Delta_{w} v\|^2_T+\|\Delta v_0\|^2_T)\leq C\|\Delta_{w} v\|^2_T.
\end{equation}
  Using Lemma \ref{norm1},  \eqref{t21}, \eqref{t11}, \eqref{3norm} and \eqref{disnorm}, gives
$$
 C_1\|v\|_{2, h}\leq \3bar v\3bar.
$$

Next, applying Cauchy-Schwarz inequality, the inverse inequality, and  the trace inequality \eqref{trace} to  \eqref{2.4new}, gives 
\begin{equation*}
    \begin{split}
  \Big| (\Delta_{w}v, \varphi)_T\Big| \leq &\|\Delta v_0\|_T \|  \varphi\|_T+
 \|v_b-v_0\|_{\partial T} \| \nabla\varphi \cdot\bn \|_{\partial T}+\|(v_n\bn_e  -\nabla v_0 )\cdot\bn\|_{\partial T} \|\varphi \|_{\partial T} \\
 \leq &\|\Delta v_0\|_T \|  \varphi\|_T+
 h_T^{-\frac{3}{2}}\|v_b-v_0\|_{\partial T} \|  \varphi\|_{  T}+h_T^{-\frac{1}{2}}\|(v_n\bn_e  -\nabla v_0 )\cdot\bn\|_{\partial T} \|\varphi  \|_{T}.
    \end{split}
\end{equation*}
This yields
$$
\| \Delta_{w}v\|_T^2\leq C( \| \Delta v_0\|^2_T  +
 h_T^{-3}\|v_b-v_0\|^2_{\partial T}+h_T^{-1}\|(v_n\bn_e  -\nabla v_0 )\cdot\bn\|^2_{\partial T}),
$$
 and further gives $$ \3bar v\3bar  \leq C_2\|v\|_{2, h}.$$

 This completes the proof of the lemma. 
 \end{proof}

\begin{remark}
Consider any $d$-dimensional polytopal element $T$. 
  There exists a hyperplane $H\subset R^d$  such that a finite number $l$ of distinct $(d-1)$-dimensional edges/faces containing $e_{i}$ are completely contained within $H$.  
 In such cases, Lemmas \ref{phi}, \ref{phi2}, and \ref{normeqva} can be proved with additional techniques. For more details, see \cite{wang}. The techniques in \cite{wang} can be readily generalized to Lemmas \ref{phi}-\ref{normeqva}.
 
 \end{remark}

\begin{theorem}
The  WG Algorithm \ref{PDWG1} has  a unique solution. 
\end{theorem}
\begin{proof}
Assume $u_h^{(1)}\in V_h$ and $u_h^{(2)}\in V_h$ are two distinct  solutions of the WG Algorithm \ref{PDWG1}.   Define  $\eta_h= u_h^{(1)}-u_h^{(2)}$. Then,  $\eta_h\in V_h^0$ and satisfies
$$
(\Delta_{w} \eta_h,  \Delta_{w} v)=0, \qquad \forall v\in V_h^0.
$$
Letting $v=\eta_h$ in the above equation   gives $\3bar \eta_h\3bar=0$. From \eqref{normeq} we have $\|\eta_h\|_{2,h}=0$, which implies
$\Delta \eta_0=0$ on each $T$,  $\eta_0=\eta_b$ and $\nabla \eta_0 \cdot\bn=\eta_n \bn_e\cdot\bn$ on each $\partial T$. 
Thus $\eta_0$ is a smooth harmonic function in $\Omega$. Using the facts $\eta_0=\eta_b$ and $\nabla \eta_0 \cdot\bn=\eta_n \bn_e\cdot\bn$ on each $\partial T$ and the boundary conditions of $\eta_b=0$ and $\eta_n\bn_e\cdot\bn=0$ on $\partial\Omega$ implies  $\eta_0=0$ and $\nabla \eta_0 \cdot\bn=0$  on $\partial\Omega$. Therefore, we obtain
$\eta_0\equiv 0$ in $\Omega$ and further  $\eta_b\equiv  0$ and $\eta_n\equiv 0$ in $\Omega$. This gives $\eta_h\equiv 0$ in $\Omega$.  Therefore, we have $u_h^{(1)}\equiv u_h^{(2)}$. 

This completes the proof of this theorem.

% Thus, $\eta_0$ is a linear function on each $T$ and $\nabla \eta_0$ is a constant vector on each element $T$.
% Using $\nabla \eta_0=\eta_n \bn_e$ on each $\partial T$ gives  $\nabla \eta_0$ is continuous over the whole domain $\Omega$. The fact that $\eta_n \bn_e\cdot bn=0$ on $\partial\Omega$ leads to $\nabla \eta_0=0$ in $\Omega$ and $\beta_g=0$ on each edge. Thus, $\eta_0$ is a constant on each element $T$. This, together with the fact $\eta_0=\eta_b$ on $\partial T$, indicates that $\eta_0$ is continuous over the whole domain $\Omega$. It follows from $\eta_b=0$ on $\partial\Omega$ that $\eta_0=0$ everywhere in the domain $\Omega$. Furthermore, $\eta_b=\eta_0=0$ on each edge. 
% This further gives $\eta_h\equiv 0$ in the domain $\Omega$. Therefore, we have $u_h^{(1)}\equiv u_h^{(2)}$. This completes the proof of this theorem.
\end{proof}

\section{Error Equations}
Let $Q_r$ denote the $L^2$ projection operator onto the finite element space consisting of piecewise polynomials of degree at most $r$.

\begin{lemma}\label{Lemma5.1}   The following property holds true, namely:
\begin{equation}\label{pro}
\Delta_{w}u =Q_r(\Delta u), \qquad \forall u\in H^2(T).
\end{equation}
\end{lemma}

\begin{proof} For any $u\in H^2(T)$,  using \eqref{2.4new} gives 
 \begin{equation*} 
  \begin{split}
 &(\Delta_{w}u, \varphi)_T\\
  =&(\Delta u,  \varphi)_T-
  \langle u|_{\partial T}-u|_T, \nabla \varphi \cdot\bn \rangle_{\partial T}+\langle (\nabla u\cdot\bn_e)|_{\partial T}\bn_e\cdot\bn -\nabla (u|_{T})\cdot\bn, \varphi\rangle_{\partial T}\\
  =&(\Delta u,  \varphi)_T=(Q_r(\Delta u),  \varphi)_T,
   \end{split}
   \end{equation*} 
  for any $\varphi\in P_r(T)$.  This  completes the proof of this lemma.
  \end{proof}

Let  $u$ be the exact solution of the biharmonic equation \eqref{model}, and let  $u_h \in V_{h}$  be its numerical approximation obtained from the Weak Galerkin   scheme \ref{PDWG1}. We define the error function, denoted by  $e_h$, as follows
\begin{equation}\label{error} 
e_h=u-u_h.
\end{equation}

\begin{lemma}\label{errorequa}
The error function $e_h$, defined in (\ref{error}), satisfies the following error equation, namely:
\begin{equation}\label{erroreqn}
(\Delta_w e_h, \Delta_w v)=\ell (u, v), \qquad \forall v\in V_h^0,
\end{equation}
where 
$$
\ell (u, v)= \sum_{T\in {\cal T}_h}-
  \langle v_b-v_0, \nabla ((Q_r-I) \Delta u)\cdot
  \bn \rangle_{\partial T}+\langle v_{n}\bn_e\cdot\bn-\nabla v_0\cdot\bn, (Q_r-I) \Delta u \rangle_{\partial T}.
$$
\end{lemma}
\begin{proof}   By utilizing \eqref{pro}, the ususal integration by parts, and setting $\varphi= Q_r \Delta u$ in \eqref{2.4new}, we obtain the following:
\begin{equation}\label{54}
\begin{split}
&\sum_{T\in {\cal T}_h} (\Delta_{w}  u, \Delta_{w} v)_T\\=&\sum_{T\in {\cal T}_h} (Q_r \Delta u, \Delta_{w} v)_T\\ 
=&\sum_{T\in {\cal T}_h} (\Delta v_0,  Q_r \Delta u)_T-
  \langle v_b-v_0, \nabla (Q_r \Delta u)\cdot
  \bn \rangle_{\partial T}\\&+\langle v_{n}\bn_e\cdot\bn-\nabla v_0\cdot\bn, Q_r \Delta u \rangle_{\partial T}\\
=& \sum_{T\in {\cal T}_h} (\Delta v_0,   \Delta u)_T-
  \langle v_b-v_0, \nabla (Q_r \Delta u)\cdot
  \bn \rangle_{\partial T}\\&+\langle v_{n}\bn_e\cdot\bn-\nabla v_0\cdot\bn, Q_r \Delta u \rangle_{\partial T}\\
=& \sum_{T\in {\cal T}_h}  (\Delta^2u, v_0)_T-\langle \nabla(\Delta u)\cdot\bn, v_0 \rangle_{\partial T}+\langle \Delta u, \nabla v_0 \cdot\bn\rangle_{\partial T}\\
&-
  \langle v_b-v_0, \nabla (Q_r \Delta u)\cdot
  \bn \rangle_{\partial T}+\langle v_{n}\bn_e\cdot\bn-\nabla v_0\cdot\bn, Q_r \Delta u \rangle_{\partial T}\\
  =&(f, v_0)+\sum_{T\in {\cal T}_h}-
  \langle v_b-v_0, \nabla ((Q_r-I) \Delta u)\cdot
  \bn \rangle_{\partial T}\\&+\langle v_{n}\bn_e\cdot\bn-\nabla v_0\cdot\bn, (Q_r-I) \Delta u \rangle_{\partial T},
\end{split}
\end{equation}
where we used \eqref{model}, $\Delta v_0\in P_{k-2}(T)$, $r=2N+k-2\geq k-2$,   $\sum_{T\in {\cal T}_h} \langle \Delta u,  v_n\bn_e\cdot\bn  \rangle_{\partial T}=\sum_{T\in {\cal T}_h}  \langle \Delta u,  v_n\bn_e\cdot\bn  \rangle_{\partial \Omega}=0$ since $v_n \bn_e\cdot\bn=0$ on $\partial \Omega$, and 
$\sum_{T\in {\cal T}_h}  \langle \nabla(\Delta u)\cdot\bn, v_b \rangle_{\partial T}= \sum_{T\in {\cal T}_h} \langle \nabla(\Delta u)\cdot\bn, v_b \rangle_{\partial \Omega}=0$ since $v_{b}=0$ on $\partial \Omega$.  

Subtracting \eqref{WG} from \eqref{54}  gives 
\begin{equation*}  
\begin{split}
&\sum_{T\in {\cal T}_h} (\Delta_{w} e_h, \Delta_{w} v)_T\\=& \sum_{T\in {\cal T}_h}-
  \langle v_b-v_0, \nabla ((Q_r-I) \Delta u)\cdot
  \bn \rangle_{\partial T}+\langle v_{n}\bn_e\cdot\bn-\nabla v_0\cdot\bn, (Q_r-I) \Delta u \rangle_{\partial T}.
\end{split}
\end{equation*}

This completes the proof of the lemma.
\end{proof}

\section{Error Estimates}

\begin{lemma}\cite{ye} 
Assume that $w$ is sufficiently regular such that $w\in H^{\max\{k+1, 4\}}(\Omega)$.    There exists a constant $C$ such that the following estimates hold true, namely:
\begin{eqnarray}\label{error1}
\Big(\sum_{T\in {\cal T}_h} h_T\|\Delta w- Q_r \Delta w\|^2_{\partial T}\Big)^{\frac{1}{2}}& \leq & C   h^{k-1}\|w\|_{k+1},
 \\
\label{error2}
\Big(\sum_{T\in {\cal T}_h} h_T^3\|\nabla(\Delta w- Q_r \Delta w)\|^2_{\partial T}\Big)^{\frac{1}{2}}& \leq & C   h^{k-1}(\|w\|_{k+1}+h\delta_{r, 0}\|w\|_4),
\end{eqnarray}
 where $\delta_{r,0}$ is the Kronecker delta such that  $\delta_{r,0}=1$ for $r=0$ and otherwise $\delta_{r,0}=0$.
 \end{lemma}

\begin{lemma}
Assume the exact solution $u$ of the biharmonic equation \eqref{model} is sufficiently regular, so that $u\in H^ {k+1} (\Omega)$. Then, there  exists a constant $C$, such that the following   estimate holds true; i.e.,
\begin{equation}\label{erroresti1}
\3bar u-Q_hu \3bar \leq Ch^{k-1}\|u\|_{k+1}.
\end{equation}
\end{lemma}
\begin{proof}
Using \eqref{2.4new},  the trace inequalities \eqref{tracein}-\eqref{trace},   and the inverse inequality,  we have,   for any $\varphi\in P_r(T)$,

\begin{equation*}
\begin{split}
&\sum_{T\in {\cal T}_h} (\Delta_{w}(u-Q_hu),  \varphi)_T\\
 = &\sum_{T\in {\cal T}_h}  (\Delta(u-Q_0u),  \varphi)_T-
  \langle Q_0u-Q_bu, \nabla \varphi \cdot\bn \rangle_{\partial T}\\&+\langle  \nabla(Q_0u) \cdot\bn-Q_n(\nabla u \cdot \bn_e)\bn_e \cdot   \bn, \varphi\rangle_{\partial T}\\
\leq &\Big(\sum_{T\in {\cal T}_h} \|\Delta(u-Q_0u)\|^2_T\Big)^{\frac{1}{2}} \Big(\sum_{T\in {\cal T}_h} \|\varphi\|_T^2\Big)^{\frac{1}{2}}\\&
 + \Big(\sum_{T\in {\cal T}_h} \| Q_0u-Q_bu\|_{\partial T} ^2\Big)^{\frac{1}{2}}\Big(\sum_{T\in {\cal T}_h}  \|\nabla \varphi \cdot\bn \|_{\partial T}^2\Big)^{\frac{1}{2}}\\
 &+ \Big(\sum_{T\in {\cal T}_h}  \|  \nabla(Q_0u) \cdot\bn-Q_n(\nabla u \cdot \bn_e)\bn_e \cdot   \bn\|_{\partial T} ^2\Big)^{\frac{1}{2}}\Big(\sum_{T\in {\cal T}_h}  \| \varphi\|_{\partial T}^2\Big)^{\frac{1}{2}}\\
\leq &\Big(\sum_{T\in {\cal T}_h} \|\Delta(u-Q_0u)\|^2_T\Big)^{\frac{1}{2}} \Big(\sum_{T\in {\cal T}_h} \|\varphi\|_T^2\Big)^{\frac{1}{2}}\\&
 + \Big(\sum_{T\in {\cal T}_h} h_T^{-1}\| Q_0u-u\|_{T} ^2+h_T\| Q_0u-u\|_{1, T} ^2\Big)^{\frac{1}{2}}\Big(\sum_{T\in {\cal T}_h} h_T^{-3} \|\varphi\|_{T}^2\Big)^{\frac{1}{2}}\\
 &+ \Big(\sum_{T\in {\cal T}_h} h_T^{-1} \|  \nabla Q_0u -  \nabla u    \|_{T} ^2+h_T\| \nabla Q_0u - \nabla u   \|_{1, T} ^2\Big)^{\frac{1}{2}}\Big(\sum_{T\in {\cal T}_h}  h_T^{-1}\| \varphi\|_{T}^2\Big)^{\frac{1}{2}}\\
&\leq Ch^{k-1}\|u\|_{k+1}\Big(\sum_{T\in {\cal T}_h} \|\varphi\|_T^2\Big)^{\frac{1}{2}}.
\end{split}
\end{equation*}
Letting $\varphi=\Delta_{w}(u-Q_hu)$ gives 
$$
\sum_{T\in {\cal T}_h} (\Delta_{w}(u-Q_hu), \Delta_{w}(u-Q_hu))_T\leq 
 Ch^{k-1}\|u\|_{k+1}\3bar u-Q_hu \3bar.
 $$  
 
 This completes the proof of the lemma.
\end{proof}

\begin{theorem}
Assume that the exact solution 
$u$
 of the biharmonic equation \eqref{model} is sufficiently regular, so that $u\in H^{\max\{k+1, 4\}} (\Omega)$. There exists a constant $C$, such that the following error estimate holds true, namely:
\begin{equation}\label{trinorm}
\3bar u-u_h\3bar \leq Ch^{k-1}(\|u\|_{k+1}+h\delta_{r,0}\|u\|_4).
\end{equation}
\end{theorem}
\begin{proof}
For the first term on the right-hand side of the error equation \eqref{erroreqn}, using Cauchy-Schwarz inequality,  the estimate \eqref{error2}, and \eqref{normeq} implies 
\begin{equation}\label{erroreqn1}
\begin{split}
&\Big|\sum_{T\in {\cal T}_h}-
  \langle v_b-v_0, \nabla ((Q_r-I) \Delta u)\cdot
  \bn \rangle_{\partial T}\Big|\\
\leq & C(\sum_{T\in {\cal T}_h} h_T^{-3}\|v_b-v_0\|^2_{\partial T} )^{\frac{1}{2}} \cdot(\sum_{T\in {\cal T}_h} h_T^3\| \nabla ((Q_r-I) \Delta u)\cdot
  \bn\|^2_{\partial T})^{\frac{1}{2}}\\
\leq & Ch^{k-1}(\|u\|_{k+1}+h\delta_{r,0}\|u\|_4)\3bar v\3bar.
\end{split}
\end{equation}

For the second term on the right-hand side of the error equation \eqref{erroreqn}, using  Cauchy-Schwarz inequality, the trace inequality \eqref{tracein}   and \eqref{normeq} gives 
\begin{equation}\label{erroreqn2}
\begin{split}
&\Big|\sum_{T\in {\cal T}_h}  \langle v_{n}\bn_e\cdot\bn-\nabla v_0\cdot\bn, (Q_r-I) \Delta u \rangle_{\partial T}\Big|\\
\leq & C(\sum_{T\in {\cal T}_h} h_T^{-1}\|v_{n}\bn_e\cdot\bn-\nabla v_0\cdot\bn\|^2_{\partial T} )^{\frac{1}{2}} \cdot(\sum_{T\in {\cal T}_h} h_T\|(Q_r-I) \Delta u \|^2_{\partial T})^{\frac{1}{2}}\\
\leq & C  \| v\|_{1,h} (\sum_{T\in {\cal T}_h} \|(Q_r-I) \Delta u \|^2_{T}+h_T^2\|(Q_r-I) \Delta u\|^2_{1,T})^{\frac{1}{2}}\\
\leq & Ch^{k-1}\|u\|_{k+1} \3bar v\3bar.
\end{split}
\end{equation}

Substituting \eqref{erroreqn1}-\eqref{erroreqn2}  into \eqref{erroreqn}  gives
\begin{equation}\label{err}
(\Delta_{w} e_h, \Delta_{w}  v)\leq   Ch^{k-1}(\|u\|_{k+1}+h\delta_{r,0}\|u\|_4) \3bar  v\3bar.
\end{equation}

By setting $v=Q_hu-u_h$ in \eqref{err},    and applying the Cauchy-Schwarz inequality along with \eqref{erroresti1}, we obtain
\begin{equation*}
\begin{split}
& \3bar u-u_h\3bar^2\\=&\sum_{T\in {\cal T}_h} (\Delta_{w}  (u-u_h), \Delta_{w}  (u-Q_hu))_T+(\Delta_{w}  (u-u_h), \Delta_{w} (Q_hu-u_h))_T\\
\leq &\Big(\sum_{T\in {\cal T}_h} \|\Delta_{w} (u-u_h)\|^2_T\Big)^{\frac{1}{2}} \Big(\sum_{T\in {\cal T}_h} \|\Delta_{w}  (u-Q_hu)\|^2_T\Big)^{\frac{1}{2}} \\& + Ch^{k-1}(\|u\|_{k+1}+h\delta_{r,0}\|u\|_4) \3bar Q_hu-u_h\3bar \\
\leq &\3bar u-u_h  \3bar\3bar u-Q_hu \3bar + Ch^{k-1}(\|u\|_{k+1}+h\delta_{r,0}\|u\|_4)  (\3bar Q_hu-u\3bar+\3bar u-u_h\3bar)  \\
\leq &C\3bar u-u_h  \3bar  h^{k-1}\|u\|_{k+1} + Ch^{k-1}(\|u\|_{k+1}+h\delta_{r,0}\|u\|_4)   h^{k-1}\|u\|_{k+1}\\&+Ch^{k-1}(\|u\|_{k+1}+h\delta_{r,0}\|u\|_4) \3bar u-u_h\3bar.
\end{split}
\end{equation*}
This gives
\begin{equation*}
\begin{split}
 \3bar u-u_h\3bar \leq & Ch^{k-1}\|u\|_{k+1}+Ch^{k-1}(\|u\|_{k+1}+h\delta_{r,0}\|u\|_4)\\\leq & Ch^{k-1}(\|u\|_{k+1}+h\delta_{r,0}\|u\|_4).
\end{split}
\end{equation*} 

This completes the proof of the theorem.
\end{proof}

\section{Error Estimates in $L^2$ Norm}
The standard duality argument is utilized to derive the $L^2$ error estimate. Recall that $e_h=u-u_h=\{e_0, e_b, \be_g\}$.  Let us denote $\zeta_h =Q_hu - u_h=\{\zeta_0, \zeta_b, \bzeta_g\}\in V_h^0$. The dual problem for the biharmonic equation \eqref{model}  seeks $w \in H_0^2(\Omega)$ satisfying 
\begin{equation}\label{dual}
\begin{split}
    \Delta^2 w&=\zeta_0, \qquad \text{in}\ \Omega,\\
w&=0,     \qquad \text{0n}\ \partial\Omega,\\
\frac{\partial w}{\partial \bn}&=0,  \qquad \text{0n}\ \partial\Omega.
    \end{split}
\end{equation}
Assume that the $H^4$-regularity property holds true; that is,
 \begin{equation}\label{regu2}
 \|w\|_4\leq C\|\zeta_0\|.
 \end{equation}
 
 \begin{theorem}
Assume that the exact solution $u$ of the biharmonic equation \eqref{model} satisfies $u\in H^{\max\{k+1, 4\}}(\Omega)$ and the $H^4$-regularity assumption  \eqref{regu2} for the dual problem \eqref{dual} holds true. Let $u_h\in V_h$ be the numerical solution of the weak Galerkin scheme \ref{PDWG1}.   Then, there exists a constant $C$ such that 
\begin{equation*}
\|e_0\|\leq Ch^{k+1-\delta_{r,0}}(\|u\|_{k+1}+h\delta_{r,0}\|u\|_4).
\end{equation*}
 \end{theorem}
 
 \begin{proof}
 Testing \eqref{dual} by $\zeta_0$, using the usual integration by parts,  we obtain
 \begin{equation}\label{e1}
 \begin{split}
 &\|\zeta_0\|^2 \\=&(\Delta^2 w, \zeta_0)\\ 
  =& \sum_{T\in {\cal T}_h} (\Delta w, \Delta \zeta_0)_T-\langle \Delta w, \nabla\zeta_0 \cdot \bn \rangle_{\partial T}+\langle  \nabla(\Delta w)\cdot \bn,  \zeta_0  \rangle_{\partial T}\\
  =& \sum_{T\in {\cal T}_h} (\Delta w, \Delta \zeta_0)_T-\langle \Delta w, \nabla\zeta_0 \cdot \bn-\zeta_n \bn_e\cdot\bn \rangle_{\partial T}+\langle  \nabla(\Delta w)\cdot \bn,  \zeta_0-\zeta_b  \rangle_{\partial T},
 \end{split}
 \end{equation}
 where we used  $\sum_{T\in {\cal T}_h}\langle \Delta w, \zeta_n\bn_e\cdot\bn \rangle_{\partial T}= \langle \Delta w,  \zeta_n\bn_e\cdot\bn \rangle_{\partial \Omega}=0$ due to $\zeta_n\bn_e\cdot\bn  =0$ on $\partial\Omega$, and $\sum_{T\in {\cal T}_h}   \langle \nabla(\Delta w)\cdot \bn,  \zeta_b   \rangle_{\partial T} =    \langle \nabla(\Delta w)\cdot \bn,  \zeta_b   \rangle_{\partial \Omega}=0$ due to $\zeta_b=0$ on $\partial\Omega$.

 Letting  $u=w$ and $v=\zeta_h$ in \eqref{54} gives
\begin{equation*}
    \begin{split}
     &  \sum_{T\in {\cal T}_h} (\Delta_{w}  w, \Delta_{w} \zeta_h)_T\\ = &  \sum_{T\in {\cal T}_h} (\Delta \zeta_0,   \Delta w)_T-
  \langle \zeta_b-\zeta_0, \nabla (Q_r \Delta w)\cdot
  \bn \rangle_{\partial T}+\langle \zeta_{n}\bn_e\cdot\bn-\nabla \zeta_0\cdot\bn, Q_r \Delta w \rangle_{\partial T},
    \end{split}
\end{equation*}
   which is equivalent to 
\begin{equation*}
    \begin{split}
&  \sum_{T\in {\cal T}_h} (\Delta \zeta_0,   \Delta w)_T\\=& \sum_{T\in {\cal T}_h} (\Delta_{w}  w, \Delta_{w} \zeta_h)_T+
  \langle \zeta_b-\zeta_0, \nabla (Q_r \Delta w)\cdot
  \bn \rangle_{\partial T}-\langle \zeta_{n}\bn_e\cdot\bn-\nabla \zeta_0\cdot\bn, Q_r \Delta w \rangle_{\partial T}.
  \end{split}
\end{equation*}
Substituting the above equation into \eqref{e1} and using \eqref{erroreqn} gives
\begin{equation}\label{e2}
 \begin{split}
   \|\zeta_0\|^2   
  = & \sum_{T\in {\cal T}_h} (\Delta_{w}  w, \Delta_{w} \zeta_h)_T+
  \langle \zeta_b-\zeta_0, \nabla ((Q_r -I)\Delta w)\cdot
  \bn \rangle_{\partial T}\\&-\langle \zeta_{n}\bn_e\cdot\bn-\nabla \zeta_0\cdot\bn, (Q_r-I) \Delta w \rangle_{\partial T}\\
  =& \sum_{T\in {\cal T}_h} (\Delta_{w} w, \Delta_{w}e_h)_T+(\Delta_{w} w, \Delta_{w} (Q_hu-u))_T-\ell(w, \zeta_h)\\ 
  =& \sum_{T\in {\cal T}_h} (\Delta_{w} Q_hw, \Delta_{w} e_h)_T+(\Delta_{w} (w-Q_hw), \Delta_{w} e_h)_T\\&+(\Delta_{w} w, \Delta_{w}(Q_hu-u))_T-\ell(w, \zeta_h)\\ 
  =&\ell(u, Q_hw) + \sum_{T\in {\cal T}_h} (\Delta_{w} (w-Q_hw), \Delta_{w} e_h)_T +(\Delta_{w} w, \Delta_{w}(Q_hu-u))_T\\&-\ell(w, \zeta_h)\\
  =&J_1+J_2+J_3+J_4.
 \end{split}
 \end{equation}
 
We will estimate the four terms $J_i(i=1,\cdots,4)$ on the last line of \eqref{e2} individually.

For $J_1$, using Cauchy-Schwarz inequality, the trace inequality \eqref{tracein},    the estimates \eqref{error1}-\eqref{error2}, gives

\begin{equation}\label{ee1}
\begin{split}
 & J_1=  \ell(u, Q_hw)\\
\leq &\Big| \sum_{T\in {\cal T}_h}-
  \langle Q_bw-Q_0w, \nabla ((Q_r-I) \Delta u)\cdot
  \bn \rangle_{\partial T}\\&+\langle Q_n(\nabla w\cdot\bn_e) \bn_e\cdot\bn-\nabla Q_0w\cdot\bn, (Q_r-I) \Delta u \rangle_{\partial T}|\\
\leq& \Big(\sum_{T\in {\cal T}_h} \| Q_bw-Q_0w \|_{\partial T}^2\Big)^{\frac{1}{2}} \Big(\sum_{T\in {\cal T}_h} \| \nabla ((Q_r-I) \Delta u)\cdot
  \bn \|_{\partial T}^2\Big)^{\frac{1}{2}} \\
&+\Big(\sum_{T\in {\cal T}_h} \|Q_n(\nabla w\cdot\bn_e) \bn_e\cdot\bn-\nabla Q_0w\cdot\bn\|_{\partial T}^2\Big)^{\frac{1}{2}} \\&\cdot\Big(\sum_{T\in {\cal T}_h} \|(Q_r-I) \Delta u \|_{\partial T}^2\Big)^{\frac{1}{2}} \\
\leq& \Big(\sum_{T\in {\cal T}_h} h_T^{-4}\| Q_bw-Q_0w \|_{  T}^2+h_T^{-2} \| Q_bw-Q_0w \|_{1, T}^2\Big)^{\frac{1}{2}} \\&\Big(\sum_{T\in {\cal T}_h} h_T^3\| \nabla ((Q_r-I) \Delta u)\cdot
  \bn \|_{\partial T}^2\Big)^{\frac{1}{2}} \\
&+\Big(\sum_{T\in {\cal T}_h}  h_T^{-2}\|Q_n(\nabla w\cdot\bn_e) \bn_e\cdot\bn-\nabla Q_0w\cdot\bn\|_{T}^2\\&+ \|Q_n(\nabla w\cdot\bn_e) \bn_e\cdot\bn-\nabla Q_0w\cdot\bn\|_{1, T}^2\Big)^{\frac{1}{2}} \cdot\Big(\sum_{T\in {\cal T}_h} h_T\|(Q_r-I) \Delta u \|_{\partial T}^2\Big)^{\frac{1}{2}} \\
\leq& \Big(\sum_{T\in {\cal T}_h} h_T^{-4}\| w-Q_0w \|_{  T}^2+h_T^{-2} \| w-Q_0w \|_{1, T}^2\Big)^{\frac{1}{2}} \\&\cdot Ch^{k-1}(\|u\|_{k+1}+h\delta_{r,0}\|u\|_4)\\
&+\Big(\sum_{T\in {\cal T}_h}  h_T^{-2}\| \nabla w -\nabla Q_0w \|_{T}^2+ \|\nabla w   -\nabla Q_0w \|_{1, T}^2\Big)^{\frac{1}{2}}  C h^{k-1} \|u\|_{k+1} \\
\leq & Ch^{k+1}(\|u\|_{k+1}+h\delta_{r,0}\|u\|_4)\|w\|_4.
\end{split}
\end{equation}

For $J_2$, using Cauchy-Schwarz inequality, \eqref{erroresti1} with $k=3$ and \eqref{trinorm} gives
\begin{equation}\label{ee2}
\begin{split}
J_2\leq \3bar w-Q_hw\3bar \3bar e_h\3bar\leq & Ch^{k-1}(\|u\|_{k+1}+h\delta_{r,0}\|u\|_4)h^2\|w\|_4\\\leq& Ch^{k+1}(\|u\|_{k+1}+h\delta_{r,0}\|u\|_4)\|w\|_4.
\end{split}
\end{equation}

  For $J_3$, denote by $Q^1$ a $L^2$ projection onto $P_1(T)$. Using \eqref{2.4} gives
  \begin{equation}\label{ee}
 \begin{split}
  &(\Delta_w (Q_hu-u), Q^1\Delta_{w} w)_T\\
  =& (Q_0u-u, \Delta (Q^1\Delta_{w}w))_T-\langle Q_bu-u, \nabla (Q^1\Delta_{w}w) \cdot\bn\rangle_{\partial T}\\&+ \langle ( Q_n(\nabla u \cdot\bn_e)\bn_e- \nabla u \cdot\bn_e \bn_e)\cdot\bn, Q^1\Delta_{w} w\rangle_{\partial T}=0,
 \end{split}
 \end{equation}
 where we used $ \Delta ( Q^1\partial^2_{ij, w} w)=0$,  $ \nabla(Q^1\partial^2_{ij, w} w)=C$,  the property of the projection opertors $Q_b$  and $Q_n$, as well as $p\geq 1, q\geq 1$.

  Using \eqref{ee}, Cauchy-Schwarz inequality, \eqref{pro} and \eqref{erroresti1}, gives 
  \begin{equation}\label{ee3}
  \begin{split}
  J_3\leq &|\sum_{T\in {\cal T}_h} (\Delta_{w} w, \Delta_{w} (Q_hu-u))_T|
  \\
  =&|\sum_{T\in {\cal T}_h} (\Delta_{w} w-Q^1\Delta_{w} w, \Delta_{w} (Q_hu-u))_T|\\
  =&|\sum_{T\in {\cal T}_h} (Q_r\Delta w-Q^1 Q_r\Delta w, \Delta_{w} (Q_hu-u))_T|\\
  \leq & \Big(\sum_{T\in {\cal T}_h} \|Q_r\Delta   w-Q^1 Q^r\Delta  w\|_T^2\Big)^{\frac{1}{2}} \3bar Q_hu-u \3bar\\
  \leq & Ch^{k+1}\|u\|_{k+1} \|w\|_4.
  \end{split}
  \end{equation}
 
For $J_4$, using Cauchy-Schwarz inequality, the trace inequality \eqref{tracein}, Lemma \ref{normeqva},   the estimates \eqref{error1}-\eqref{error2},  \eqref{erroresti1}, \eqref{trinorm} gives
\begin{equation}\label{ee4}
\begin{split}
J_4=&\ell(w, \zeta_h)\\
\leq &\Big|\sum_{T\in {\cal T}_h} -
  \langle \zeta_b-\zeta_0, \nabla ((Q_r-I) \Delta w)\cdot
  \bn \rangle_{\partial T}\\&+\langle \zeta_{n}\bn_e\cdot\bn-\nabla \zeta_0\cdot\bn, (Q_r-I) \Delta w \rangle_{\partial T} \Big| \\
\leq& \Big(\sum_{T\in {\cal T}_h} h_T^{-3}\|\zeta_b-\zeta_0\|_{\partial T}^2\Big)^{\frac{1}{2}} \Big(\sum_{T\in {\cal T}_h} h_T^3\|\nabla ((Q_r-I) \Delta w)\cdot
  \bn\|_{\partial T}^2\Big)^{\frac{1}{2}}   \\
&+\Big(\sum_{T\in {\cal T}_h} h_T^{-1}\|\zeta_{n}\bn_e\cdot\bn-\nabla \zeta_0\cdot\bn\|_{\partial T}^2\Big)^{\frac{1}{2}} \Big(\sum_{T\in {\cal T}_h} h_T\|(Q_r-I) \Delta w \|_{\partial T}^2\Big)^{\frac{1}{2}}  \\
 \leq& Ch^{2-\delta_{r,0}}\|w\|_{4}\3bar\zeta_h\3bar   \\
\leq &Ch^{2-\delta_{r,0}}\|w\|_{4}(\3bar u-u_h\3bar+\3bar u-Q_hu\3bar)  \\
\leq &Ch^{k+1-\delta_{r,0}}\|w\|_4(\|u\|_{k+1}+h\delta_{r,0}\|u\|_4).
\end{split}
\end{equation}

Using \eqref{regu2} and substituting \eqref{ee1}-\eqref{ee2} and \eqref{ee3}-\eqref{ee4}  into \eqref{e2} gives
$$
\|\zeta_0\|^2\leq Ch^{k+1-\delta_{r,0}}\|w\|_4(\|u\|_{k+1}+h\delta_{r,0}\|u\|_4)\leq Ch^{k+1-\delta_{r,0}} (\|u\|_{k+1}+h\delta_{r,0}\|u\|_4)\|\zeta_0\|.
$$
This gives
$$
\|\zeta_0\|\leq Ch^{k+1-\delta_{r,0}} (\|u\|_{k+1}+h\delta_{r,0}\|u\|_4),
$$
which, using the triangle inequality, gives
$$
\|e_0\|\leq \|\zeta_0\|+\|u-Q_0u\|\leq Ch^{k+1-\delta_{r,0}}(\|u\|_{k+1}+h\delta_{r,0}\|u\|_4). 
$$

This completes the proof of the theorem. 
\end{proof}

\end{document}